\documentclass[12pt]{amsart}
\usepackage{amssymb}
\usepackage{amsmath}
\usepackage{amscd}
\usepackage[dvips]{graphicx}
\usepackage{book tabs}
\usepackage{hyperref}
\usepackage{dsfont}
\usepackage{float}

\numberwithin{equation}{section}
\newtheorem{thm}{Theorem}
\newtheorem{prop}[thm]{Proposition}
\newtheorem{lemma}[thm]{Lemma}
\newtheorem{cor}[thm]{Corollary}

\newtheorem{assum}[thm]{Assumption}

\theoremstyle{definition}

\newtheorem{example}[thm]{Example}
\newtheorem{remark1}[thm]{Remark}
\newtheorem{openproblem1}[thm]{Open problem}

\numberwithin{thm}{section}

\newcounter{FNC}[page]
\def\newfootnote#1{{\addtocounter{FNC}{2}$^\fnsymbol{FNC}$%
     \let\thefootnote\relax\footnotetext{$^\fnsymbol{FNC}$#1}}}

\newcommand{\bs}{\backslash}

\newcommand{\R}{\mathbb{R}}

\renewcommand{\P}{\mathbb{P}}

\newcommand{\VV}{\mathcal{V}}
\newcommand{\HH}{\mathcal{H}}

\newcommand{\st}{\mathrm{s.t.}}

\DeclareMathOperator{\linspan}{span}
\DeclareMathOperator{\kernel}{ker}

\DeclareMathOperator{\qm}{QM}
\newcommand{\ideal}{\mathrm{I}}
\DeclareMathOperator{\pos}{pos}

\usepackage{xcolor}

\title[Multilinear Programming]{
A Semidefinite Hierarchy for Disjointly Constrained Multilinear Programming}

\author{Kai Kellner}

\address{Goethe-Universit\"at, FB 12 -- Institut f\"ur Mathematik,
Postfach 11 19 32, D--60054 Frankfurt am Main, Germany}

\email{kellner@math.uni-frankfurt.de}

\begin{document}

\begin{abstract}
Disjointly constrained multilinear programming concerns the problem of 
maximizing a multilinear function on the product of finitely many disjoint 
polyhedra.
While maximizing a linear function on a polytope (linear programming) is known 
to be solvable in polynomial time, even bilinear programming is NP-hard.
Based on a reformulation of the problem in terms of sum-of-squares polynomials, 
we study a hierarchy of semidefinite relaxations to the problem.
It follows from the general theory that the sequence of optimal values 
converges asymptotically to the optimal value of the multilinear program.
We show that the semidefinite hierarchy converges generically in finitely many 
steps to the optimal value of the multilinear problem.

We outline two applications of the main result.
For nondegenerate bimatrix games, a Nash equilibrium can be computed by the 
sum of squares approach in finitely many steps.
Under an additional geometric condition, the NP-complete containment problem 
for projections of $\HH$-polytopes can be decided in finitely many steps.
\end{abstract}

\maketitle

\section{Introduction}

A disjointly constrained multilinear programming problem is the optimization
problem of maximizing a multilinear function on the product of finitely many
disjoint polyhedra.
To be more precise, let $l\ge 2$ and  $f(x_1,x_2,\ldots,x_l)$ be a real-valued 
function linear in the variable tuples
$x_i = (x_{i,1},\ldots,x_{i,d_i})$ for $i=1,\ldots,l$.
Consider nonempty polyhedra $P_i\subseteq\R^{d_i}$ as given by the intersection 
of finitely many halfspaces and an affine subspace.
Then a (disjointly constrained) \emph{multilinear program} has the form
\begin{align}
\begin{split} \label{eq:multilinear}
  f^* := \max\ &\ f(x_1,x_2,\ldots,x_l) \\
  \st\ &\ (x_{1},\ldots,x_{l})\in P_1 \times\cdots\times P_l \, .
\end{split}
\end{align}

If $l=1$, this reduces to linear programming. Throughout the paper we assume
$l\ge 2$. 
For $l=2$, problem~\eqref{eq:multilinear} is known as \emph{bilinear
programming} or indefinite quadratic optimization.
This NP-hard problem lies in-between linear and quadratic optimization.
Originally, bilinear programming has been considered as a generalization of
bimatrix games in game theory.
Today it is motivated by a various of applications, including 
(constrained) bimatrix games~\cite{mangasarian1964two}, 
quantum information theory~\cite{berta2015},
and geometric containment problems~\cite{kellner2015,kt}.
Prominent approaches to solve bilinear and multilinear programs are based on
vertex tracking algorithms~\cite{Gallo1977,Konno1976}.

In this note we state a semidefinite hierarchy to multilinear programs based on 
the concept of sum of squares; see, e.g., the books of Blekherman et. 
al.~\cite{Blekherman2013}, or Lasserre~\cite{lasserre2015}, or the survey by 
Laurent~\cite{laurent2009}.
To that end the problem is rewritten as a program with a nonnegativity 
constraint. This constraint is then replaced by a weaker condition, namely 
being a sum of polynomial squares which can be formulated as a semidefinite 
constraint. 
A particular question is whether the multilinear structure can be used in sum
of squares certificates in order to state convergence statements.
Indeed, we show finite convergence in generic cases by using an extension of 
Putinar's Positivstellensatz stated by Marshall~\cite{marshall2009}.
This has impact on various applications including polyhedral containment and 
game theory.

It is well-known in game theory that a \emph{bimatrix game} (cf. 
Section~\ref{sec:game}) has a formulation as a disjointly constrained bilinear 
optimization problem.
The sum of squares approach to (polynomial) games has been considered by 
Laraki and Lasserre~\cite{Laraki2010}, and Parrilo~\cite{parrilo2006}.
They did not establish finite convergence.
By our main result a Nash equilibrium of a nondegenerate bimatrix game can be 
computed via sums of squares in finitely many steps.
This can be extended to finite $l$-person games.

In~\cite{kt} Theobald and the author formulated the problem of whether an 
$\HH$-polytope is contained in a $\VV$-polytope as a bilinear programming 
problem and proved finite convergence for the sum of squares hierarchy under 
some conditions.
In Section~\ref{sec:contain} we extend this result to the NP-complete 
\emph{projective polyhedral containment problem}, the problem of whether the 
(coordinate) projection of an $\HH$-polyhedron is contained in the projection 
to another $\HH$-polyhedron~\cite{kellner2015}.
Note that the projection of an $\HH$-polyhedron is again a polyhedron but a 
projection-free representation is not easy to achieve.

The paper is structured as follows.
In Section~\ref{sec:bilinear} we outline some basic geometric behavior of
multilinear programs.
A semidefinite hierarchy based on sum of squares and a proof of its
convergence in well-defined cases is given in Section~\ref{sec:sos}.
In Section~\ref{sec:apps} we sketch the application of our main result to 
polyhedral containment and game theory.

\section{Multilinear Programming} \label{sec:bilinear}

Note that a multilinear function $f=f(x_1,x_2,\ldots,x_l)$ can be written as
\begin{align*}
	f &= \sum_{\emptyset\neq L\subseteq [l]} Q^{(L)}(x_i \,:\, i\in L) 
	= \sum_{j_1=1}^{d_1} \cdots \sum_{j_l=1}^{d_l}
	Q^{(1,\ldots,l)}_{j_1\ldots,j_{l}} x_{1,j}\cdots x_{l,j}
	+ \dots + \sum_{i=1}^l \sum_{j_i=1}^{d_i} Q^{(i)}_{j_i} x_{i,j_i}
\end{align*}
for tensors $Q^{(L)}\in\R^{\times\{d_i\ |\ i\in L\}}$ with 
$\emptyset\neq L\subseteq [l]:=\{1,\ldots,l\}$
and $Q^{[l]}\neq 0$.
For $i\in [l]$ let $a^{(i)}\in\R^{m_i}, A^{(i)}\in\R^{m_i\times d_i},\
b^{(i)}\in\R^{n_i},\ B^{(i)}\in\R^{n_i\times d_i}$.
Define the polytopes
\begin{equation*}
  P_i = \left\{ x_i \in\R^{d_i}\ |\ a^{(i)}\ge A^{(i)}x_i ,\ b^{(i)} = 
B^{(i)} x_i\right\} .
\end{equation*}
We set $\P_{1,\ldots,l} := P_1\times\ldots\times P_l$ 
and often use $\P$ if there is no risk of confusion.

Throughout the paper we make the following assumptions on the representations
of the polytopes $P_i$.

\begin{assum}\label{ass1}
For $i\in [l]$: 
\begin{enumerate}
  \item
The equality constraints are linearly independent, i.e., the dimension of the
affine subspace $\mathcal{B}_i := \{x\in\R^{d_i}\ |\ b^{(i)} = B^{(i)} x_i\}$ 
equals $d_i-n_i$ and $d_i-n_i >0$.
  \item
$\{ x\in\R^{d_i}\ |\ a^{(i)}\ge A^{(i)} x \}$ is full-dimensional in the
affine subspace $\mathcal{B}_i$, i.e., 
$\dim\{x\in\mathcal{B}\ |\ a^{(i)}\ge A^{(i)} x \}=d_i-n_i$.
\end{enumerate}
\end{assum}

The first assumption can be achieved by Gaussian elimination.
The second assumption particularly implies nonemptyness of $P_1,\ldots,P_l$. 
It can be tested (and achieved) by linear programming methods in polynomial 
time.
We understand \emph{relative interior} of a polytope $P_i$ as with respect to 
the affine subspace $\mathcal{B}_i$.

Given a polytope $P_i$, the outer normal cone $C_i$ of a boundary point $v$ of 
$P_i$ is the cone in $\linspan\{(B^{(i)})^T\}$ generated by the outer normals 
of the constraints that are active at that point, i.e., 
$C_i = \pos\{ (A^{(i)}_j)^T \ |\ a^{(i)}_j = A^{(i)}_i v ,\ j\in [m_i] \} 
\subseteq \linspan\{(B^{(i)})^T\}$.
Note that the definition of a normal cone here differs slightly from the usual 
one as we restrict it to the subspace $\linspan\{(B^{(i)})^T\}$.
The usual normal cone is given by $C_i+\linspan\{(B^{(i)})^T\}$.

The boundary of the product polytope
$\partial(\P)=\partial(P_1\times\ldots\times P_l)$ equals the product of the
boundaries $ \partial P_1 \times\ldots\times \partial P_l$.
Moreover, $F$ is a \emph{face} of $P_1\times\ldots\times P_l$ if and only if 
$F=F_1 \times\ldots\times F_l$ for faces $F_i$ of $P_i$.
In particular, any vertex of $\P$ is given by a tuple of
vertices of $P_1 ,\ldots, P_l$.

We state a basic result on multilinear programming. For a proof we refer to
Konno~\cite{Konno1976}; see also Drenick~\cite{drenick1992}.

\begin{prop} \label{prop:bilin1}
Let $P_1 ,\ldots, P_l$ be nonempty polytopes. 
Then the set of optimal solutions to~\eqref{eq:multilinear} is part of the
boundary and the maximum is attained at a tuple of vertices of 
$P_1 ,\ldots,P_l$. 
\end{prop}

We state two corollaries that will be used in the next section.

\begin{cor}\label{cor:optface}
The set of optimal solutions to~\eqref{eq:multilinear} is an union of proper
faces $F=F_1 \times\ldots\times F_l$ of $\P$.
\end{cor}

\begin{proof}
Let $(\bar{x}_1,\ldots,\bar{x}_l)$ be an optimal solution.
By Proposition~\ref{prop:bilin1} 
$(\bar{x}_1,\ldots,\bar{x}_l)\in\partial(P_1\times\ldots P_l)$.
Let $F_{\bar{x}_i}$ be the minimal face containing $\bar{x}_i$.
Then for any 
$(x_1,\ldots,x_l)\in F_{\bar{x}_1} \times\ldots\times F_{\bar{x}_l}$, we have
$f(\bar{x}_1,\ldots,\bar{x}_l)=f(x_1,\ldots,\bar{x}_i,\ldots,x_l)$ by the
linearity of $f$ in every $x_i=(x_{i,1},\ldots,x_{i,d_i}),\ i\in [l]$.
\end{proof}

The second corollary follows directly from the latter one.

\begin{cor} \label{cor:finsol}
Problem~\eqref{eq:multilinear} has finitely many optimal solutions if and only
if every optimal solution of the problem is a vertex of $\P$.
\end{cor}

\section{A Semidefinite hierarchy based on sum of squares}
\label{sec:sos}

In this section, we apply sum of squares techniques to multilinear programming
problems.
Our main goal is to show that the outcoming hierarchy of semidefinite programs
converges after finitely many steps in generic cases (in a well-defined
sense); see Theorem~\ref{thm:fin-convergence}.

The multilinear problem~\eqref{eq:multilinear} can equivalently be stated as a
nonnegativity question
\begin{align*}
  f^* &= \max \left\{ f(x_1,\ldots,x_l)\ |\ x_i\in P_i,\ i\in [l] \right\} \\
  &= \min \left\{\mu\ |\ \mu-f(x_1,\ldots,x_l)\geq 0 
     \ \text{ for }\ (x_1,\ldots,x_l)\in \P \right\} .
\end{align*}
Replacing the nonnegativity condition by sum of squares is a common method in
polynomial optimization.
A polynomial $p(x)\in\R[x]$ is a \emph{sum of squares (sos)} if and only if
there exists a \emph{positive semidefinite matrix $M$} of appropriate size
such that \emph{$p(x) = [x]_t^T M [x]_t$}, where $[x]_t$ is the vector of
monomials up to degree $t$ for some nonnegative integer $t$.
We denote the set of all sos-polynomials in the variables $x$ by
$\Sigma[x]\subseteq\R[x]$.
While every sos-polynomial is nonnegative, the converse is not true in general.
Thus replacing nonnegativity by being a sum of squares in the above problem
(where the equality constraints have to be treated separately) yields a upper
bound for $f^*$,
\begin{align*}
 f^*  &\leq \min \left\{\mu\ |\ \mu-f(x_1,\ldots,x_l) \in\qm+I\right\} ,
\end{align*}
where
\[
  \qm = \left\{ \sigma_0 + \sum_{i=1}^l \sum_{j=1}^{m_i} \sigma_{ij}\,
(a^{(i)}-A^{(i)} x_i)_j \ |\ \sigma_{ij} \in\Sigma[x] \right\}
\]
is the quadratic module generated by the linear inequality constraints and
\[
  \ideal = \left\{ \sum_{i=1}^l \sum_{j=1}^{n_i} \tau_{ij} 
  (b^{(i)}-B^{(i)}x_i)_j \ |\ \tau_{i,j} \in\R[x] \right\}
\]
is the ideal generated by the linear equality constraints.

The idea is to test membership of $\mu-f(x_1,\ldots,x_l)$ in truncations of the
quadratic module and the ideal.
Testing membership in these sets can be done by deciding feasibility of a
certain semidefinite inequality system, which is solvable in polynomial time if 
the degrees are fixed. 
To be precise, define the $t$-th truncation as 
\[
  \qm_{t} = \left\{ \sigma_0 + \sum_{i=1}^l \sum_{j=1}^{m_i} \sigma_{ij}\,
(a^{(i)}-A^{(i)} x_i)_j \ |\ \sigma_{ij} \in\Sigma_{2t-2}[x]
\right\}
\]
and
\[
  \ideal_{t} = \left\{\sum_{i=1}^l \sum_{j=1}^{n_i} \tau_{ij}
(b^{(i)}-B^{(i)}x_i)_j \ |\ \tau_{i,j} \in\R_{2t-1}[x] \right\}.
\]
The $t$-th sos program is then 
\begin{align}\label{eq:putinar}
  f_t\ &:=\ \inf \left\{\mu\ |\ \mu-f(x_1,\ldots,x_l) \in\qm_{t}+I_{t} \right\} 
\end{align}
with $f^* \leq\ldots\leq f_{l+1}\le f_l$.
Note that the first relaxation order making sense is $t=l$ as the degree of
$f$ is $l$.

We state our main result.

\begin{thm} \label{thm:fin-convergence}
Let $P_1,P_2,\ldots,P_l$ be polytopes with nonempty relative interior 
satisfying Assumption~\ref{ass1}.
\begin{enumerate}
  \item
The sequence $(f_t)_t$ of optimal values~\eqref{eq:putinar} converges
asymptotically from above to the optimal value $f^*$.
  \item
If the multilinear programming problem~\eqref{eq:multilinear} has only
finitely many optimal solutions, then $f^* -f\in\qm+\ideal$.
Thus the sequence of the optimal values $(f_t)_t$ {converges in finitely many 
steps} to the optimal value $f^*$.
\end{enumerate}
\end{thm}

The proof of Theorem~\ref{thm:fin-convergence} extends the one for the
containment problem of an $\HH$-polytope in a $\VV$-polytope in~\cite{kt}.
We start with the proof of part (1) of the statement as it follows directly
from the general theory.

\begin{prop}[{Putinar's Positivstellensatz~\cite{putinar1993}}]
\label{prop:putinar}
Let $S=\{x\in\R^d\ |\ g(x)\geq 0\ \forall g\in G,\ h(x)=0\ \forall h\in H\}$
for some finite subsets $G,H\subseteq\R[x]$.
If the quadratic module 
$\qm(G) 
= \{\sigma_0 + \sum_{g\in G} \sigma_g g\ |\ \sigma_0,\sigma_g \in\Sigma[x] \}$ 
is Archimedean, then the sum of the quadratic module with the ideal generated
by $H$, $\qm(G)+I(H)$, contains every polynomial $p\in\R[x]$ positive on $S$.
\end{prop}

A quadratic module is called \emph{Archimedean} if $N-x^T x\in\qm$ for some
positive integer $N$.
Archimedeaness of the quadratic module implies compactness of the set $S$, but
the converse is generally not true.
In our specific situation with only linear constraint polynomials defining a
polytope, this condition is always 
satisfied~\cite[Theorem 7.1.3]{marshall2008}.

\begin{proof}[Proof of Theorem~\ref{thm:fin-convergence} (1)]
By Proposition~\ref{prop:putinar}, $f^*-f+\varepsilon\in\qm+\ideal$ for every
$\varepsilon>0$.
The claim follows by letting $\varepsilon$ tend to zero.
\end{proof}

Proving part (2) is more involved.
Our goal is to apply Marshall's Nichtnegativstellensatz, an extension of 
Putinar's result.
To that end, we introduce some necessary notation.
See Nie~\cite{nie2012} for a similar application of Marshall's result to
general polynomial optimization problems.

Given a variety $V\subseteq\R^n$ of dimension $d$, denote by $I(V)$ the 
vanishing ideal of $V$ and by $R[V]=\R[x]/I(V)$ the coordinate ring of $V$.

\begin{prop}[{Marshall's Nichtnegativstellensatz~\cite[Theorem
9.5.3]{marshall2008}}] \label{prop:marshall}
Let $V\subseteq\R^n$ be a
variety of dimension $d$ and let $f,g_1,\ldots,g_m\in\R[V]$ with 
$S = \{ x\in V\ |\ g_i(x)\ge 0,\ i\in [m]\}$.
Assume that the quadratic module $M$ generated by $g_1,\ldots,g_m$ is
Archimedean in $\R[V]$.
Further suppose that for any global maximizer $\bar{x}$ of $f$ on $S$ the 
following holds:
\begin{enumerate}
  \item
  $\bar{x}$ is a nonsingular point of $V$,
  \item
  there exist local parameters $t_1,\ldots,t_d$ and an index set $I$ of
  size $1\le k\le d$ such that $t_i=g_i$ for $i\in I$,
  \item
  the linear part of $f$ in the localizing parameters equals
  $\sum_{i\in I} c_i t_i$ for some negative coefficients $c_i$, $i\in I$, and
  \item
  the quadratic part of $f$ in the localizing parameters is negative definite
  on the $t_i=0$ for $i\in I$.
\end{enumerate}
Then $f_{\max}-f\in M$, where $f_{\max}$ denotes the global maximum of $f$ 
on $S$.
\end{prop}

The last two conditions in Proposition~\ref{prop:marshall} are called
boundary Hessian condition (BHC); see~\cite{marshall2008,marshall2009}.
To prove Theorem~\ref{thm:fin-convergence} we need one more lemma.

\begin{lemma} \label{lem:redundant}
Let $P_1,\ldots,P_l$ be nonempty polytopes and $c_0 - c^T x_i \ge 0$ be a
redundant inequality in the representation of some $P_i$. Then it is also
redundant in the hierarchy~\eqref{eq:putinar}.
\end{lemma}

The lemma follows by a simple application of the affine version of Farkas'
Lemma; see e.g. \cite{kt}.

\begin{proof}[Proof of Theorem~\ref{thm:fin-convergence} (2)]
Let $ \bar{x}=(\bar{x}_1,\ldots,\bar{x}_l)\in \P = P_1\times\ldots\times P_l$
be an arbitrary but fixed optimal solution to~\eqref{eq:multilinear}.
As by Assumption~\ref{ass1} the linear equality constraints are linearly 
independent for each $P_i$, the rank of the gradient equals the number of 
equations for every point in the variety.
Thus every point is nonsingular (cf. \cite[Theorem 9]{cox2008}).

Since $f$ is linear in $x_i$ for any $i\in [l]$, its derivative w.r.t. $x_i$ 
has the form
\[
\frac{\partial}{\partial x_i} f(x_1,\ldots,x_l) 
= \sum_{\emptyset\neq L\subseteq [l],\ i\in L} Q^{(L)}(x_j\,:\, j\in
L\bs\{i\}) \,,
\]
a $d_i$-dimensional vector with entries in the polynomial ring 
$\R[x_1,\ldots,x_{i-1},x_{i+1},\ldots,x_{l}]$.
Every component is a multilinear function in
$(x_1,\ldots,x_{i-1},x_{i+1},\ldots,x_{l})$.

Fix $i\in [l]$.
Assume $\frac{\partial}{\partial x_i} f(\bar{x})$ lies in the sum of the 
outer normal cone (in $\linspan\{(B^{(i)})^T\}$) of an at least one-dimensional 
face $F_i$ of $P_i$ and the linear subspace $\linspan\{(B^{(i)})^T\}$.
By Corollary~\ref{cor:optface}, we have 
\[
	\left(\frac{\partial}{\partial x_i} f(\bar{x})\right) (\hat{x}_i) 
	= \max\left\{ \left(\frac{\partial}{\partial x_i} f(\bar{x})\right) (x_i)
	\ |\ x_i\in P_i \right\}
	= \left(\frac{\partial}{\partial x_i} f(\bar{x})\right) (\bar{x}_i) 
\]
for every $\hat{x}_i\in F_i$ and thus
\begin{align*}
f(\bar{x}) &= \sum_{\emptyset\neq L\subseteq [l],\ i\in L} 
Q^{(L)}(\bar{x}_j\,:\, j\in L)
+ \sum_{\emptyset\neq L\subseteq [l],\ i\not\in L} 
Q^{(L)}(\bar{x}_j\,:\, j\in L) \\
&= \left(\frac{\partial}{\partial x_i} f(\bar{x})\right) (\bar{x}_i) 
+ \sum_{\emptyset\neq L\subseteq [l],\ i\not\in L} 
Q^{(L)}(\bar{x}_j\,:\, j\in L) \\
&= \left(\frac{\partial}{\partial x_i} f(\bar{x})\right) (\hat{x}_i) 
+ \sum_{\emptyset\neq L\subseteq [l],\ i\not\in L} 
Q^{(L)}(\bar{x}_j\,:\, j\in L) \\
&= f(\bar{x}_1,\ldots,\hat{x}_i,\ldots,\bar{x}_l) \,,
\end{align*}
in contradiction to Corollary~\ref{cor:finsol} and the assumption in part (2)
of the theorem.
Hence $\frac{\partial}{\partial x_i} f(\bar{x})$ lies in the sum of the 
outer normal cone $C_i$ of $\bar{x}_i$ and the subspace 
$\linspan\{(B^{(i)})^T\}$.

As $\bar{x}_i$ is a vertex of $P_i$, $C_i$ is full-dimensional in the 
$(d_i-n_i)$-dimensional linear subspace $\linspan\{(B^{(i)})^T\}$. 
Thus there exist $d_i-n_i$ points in the outer normal cone such that
$\frac{\partial}{\partial x_i} f(\bar{x})$ is a (strictly) positive
combination of these points.
The points correspond to redundant inequalities in the representation of $P_i$ 
that are active in $\bar{x}_i$.
As by Lemma~\ref{lem:redundant} redundant constraints do not affect (existence
and degree of) the sum of squares certificates, we can assume that the
representation of $P_i$ is given in a form such that there exists an index set
$I_i\subseteq [m_i]$ of cardinality $|I_i|=d_i-n_i$ corresponding to linearly
independent, active constraints and there exist 
$\alpha^{(i)}\in\R^{m_i}$ and $\beta^{(i)}\in\R^{n_i}$ with
\begin{equation}\label{eq:poscombi}
	\frac{\partial}{\partial x_i} f(\bar{x}) = (A^{(i)})^T \alpha^{(i)} +
	(B^{(i)})^T \beta^{(i)} \,\quad \alpha^{(i)}_j > 0 \text{ for } j\in I_i 
	\text{ and } \alpha^{(i)}_j = 0 \text{ for } j\not\in I_i \,.
\end{equation}

Let $I_1,\ldots,I_l$ be index sets as constructed in the previous paragraph.
W.l.o.g. set $I_i = [d_i-n_i]$ for $i\in[l]$.
Consider the affine variable transformation
\[
  \phi :\ \R^{d_1+\dots+d_l}\to\R^{d_1+\dots+d_l} ,\quad (x_1,\ldots,x_l)
\mapsto \begin{bmatrix} (a^{(1)}-A^{(1)} x_1)_{I_1} \\ b^{(1)}-B^{(1)} x_1 \\
\vdots \\ (a^{(l)}-A^{(l)} x_l)_{I_l} \\ b^{(l)}-B^{(l)} x_l \end{bmatrix}
=: \begin{pmatrix} s_1 \\ \vdots \\ s_l \end{pmatrix} .
\]
The new variables 
$ (s_1,\ldots,s_l) := (s_{11},\ldots,s_{1d_1}\ldots,s_{l1},\ldots,s_{ld_l})$
serve as localizing parameters on the variety defined by 
$s_{ij}=(b^{(i)}-B^{(i)}x_i)_j = 0,\ j\in [d_i]\bs I_i,\ i\in [l] $
as in Proposition~\ref{prop:marshall}.
The inverse of $\phi$ is given by 
\[
   \phi^{-1} :\ \R^{d_1+\dots+d_l}\to\R^{d_1+\dots+d_l} ,\quad
  (s_1,\ldots,s_l)\mapsto 
  \begin{bmatrix}
  \begin{bmatrix} A^{(1)}_{I_1} \\[+1ex] B^{(1)} \end{bmatrix}^{-1} \cdot 
  \left( \begin{pmatrix} a^{(1)}_{I_1} \\[+1ex] b^{(1)} \end{pmatrix}
  -s_1\right)
  \\ \vdots \\ 
  \begin{bmatrix} A^{(l)}_{I_l} \\[+1ex]  B^{(l)} \end{bmatrix}^{-1} \cdot 
  \left( \begin{pmatrix} a^{(l)}_{I_l} \\[+1ex] b^{(l)} \end{pmatrix}
  -s_l\right)
  \end{bmatrix} .
\]

Setting 
$ \bar{A}^{(i)} = [A^{(i)}_{I_i}, B^{(i)}]^T ,\ 
  \bar{a}^{(i)} = (a^{(i)}_{I_i},b^{(i)})^T $,
the objective function $f$ has the form
\begin{align*}
  f\circ\Phi^{-1}(s_1,\ldots,s_l) 
  &= \sum_{\emptyset\neq L\subseteq [l]} Q^{(L)} \left( (\bar{A}^{(i)})^{-1}
\left(\bar{a}^{(i)}-s_i\right) \,:\, i\in L \right) \\
  &= \sum_{\emptyset\neq L\subseteq [l]} \sum_{M\subseteq L} 
  Q^{(L)} \left( \{ -(\bar{A}^{(i)})^{-1} s_i \,:\, i\in M \}\cup
  \{ (\bar{A}^{(i)})^{-1} \bar{a}^{(i)} \,:\, i\in L\bs M \} \right)
\end{align*}
in the parameterization space.
The homogeneous degree-1 part of $f$ in $(s_1,\ldots,s_l)$ is
\[
  f_1 := \sum_{\emptyset\neq L\subseteq [l]} \sum_{i\in L} 
  Q^{(L)} \left( \{ -(\bar{A}^{(i)})^{-1} s_i \}\cup
  \{ (\bar{A}^{(j)})^{-1} \bar{a}^{(j)} \,:\, j\neq i \} \right) \,.
\]

In order to show that $f$ satisfies the conditions in 
Proposition~\ref{prop:marshall}, we show that the gradient of $f_1$ w.r.t 
$s_1,\ldots,s_l$ equals 
$(-\alpha^{(1)}_{I_1}-\beta^{(1)},\ldots,-\alpha^{(l)}_{I_l}-\beta^{(l)})$.
Then $f_1$ has only negative coefficients on the variety defined by the
equality constraints, i.e.,
\[
  f_1 = (\alpha^{(1)}_{I_1})^T (s_1)_{I_1} + \dots 
+ (\alpha^{(l)}_{I_l})^T (s_l)_{I_l} 
  \quad\text{on}\quad (s_i)_{[d_i]\bs I_i}=0,\ i\in [l] \,,
\] 
and thus the third condition of Proposition~\ref{prop:marshall} is satisfied.
Since $|I_1|+\dots+|I_l|=d_1-n_1+\ldots+d_l-n_l$ equals the dimension of the
variety defined by the equality constraints in $P_1\times\dots\times P_l$, the
last condition is obsolete.
Marshall's Nichtnegativstellensatz then implies 
$f^* -f(x_1,\ldots,x_l) \in\qm+\ideal$.

It remains to show the desired equality for the gradient of $f_1$.
To that end note that
\[
  (\bar{x}_1,\ldots,\bar{x}_l)=\phi^{-1}(0) 
  = ((\bar{A}^{(1)})^{-1}\bar{a}^{(1)}, \ldots,
  (\bar{A}^{(l)})^{-1}\bar{a}^{(l)}) \,.
\]
We have
\begin{align*}
  \frac{\partial}{\partial s_i} f_1 (0) 
  &= -\sum_{\emptyset\neq L\subseteq [l],\ i\in L} 
  ((\bar{A}^{(i)})^{-1})^T 
  Q^{(L)}( (\bar{A}^{(j)})^{-1} \bar{a}^{(j)}\,:\, j\in L\bs\{i\}) \\
  &= - ((\bar{A}^{(i)})^{-1})^T \frac{\partial}{\partial x_i} f(\bar{x}) \\
  &= - \alpha^{(i)}_{I_i} - \beta^{(i)} \,,
\end{align*}
where the last equation follows from~\eqref{eq:poscombi}.
This completes the proof.
\end{proof}

\section{Applications}\label{sec:apps}

\subsection{Geometric containment problems}\label{sec:contain}
As outlined in the introduction, this note is motivated by geometric 
containment problems.
More precisely, Theobald and the author studied the classical problem 
(cf.~\cite{freund1985,gritzmann1993,gritzmann1994}) to decide whether a 
given $\HH$-polytope is contained in a given $\VV$-polytope. 
While the converse containment problem, $\VV$-in-$\HH$, is solvable in 
polynomial time, the $\HH$-in-$\VV$ problem is co-NP-complete.
The main result in~\cite{kt} is the formulation as a certain bilinear 
programming problem and the proof of finite convergence in some cases. 
Theorem~\ref{thm:fin-convergence} extends this results to disjointly 
constrained multilinear programs.

In~\cite{kellner2015} the author gave a bilinear formulation of the projective 
polyhedral containment problem:
For $a\in\R^k$ and $b\in\R^{l}$ let
\begin{align*}
  P = \left\{(x,y)\in\R^{d+m}\ |\ a\ge Ax+A'y \right\} 
  \quad\text{and}\quad 
  Q = \left\{(x,y')\in\R^{d+n}\ |\ b\ge Bx+B'y' \right\}
\end{align*}
be nonempty polyhedra and consider their projection to $\R^d$, denoted by 
$\pi(P)$ and $\pi(Q)$.
The projective polyhedral containment problem asks whether 
$\pi(P)\subseteq\pi(Q)$.
Note that the projection of a polyhedron is again a polyhedron but a 
projection-free description is not easy to achieve.
Indeed, while deciding containment of $P$ in $Q$ is solvable in polynomial 
time, it is shown in \cite[Theorem 3.2.4]{kphd} that deciding containment of 
$\pi(P)$ in $\pi(Q)$ is co-NP-complete.
This hardness statement holds true for $m=0$.

\begin{prop}\cite[Theorem 3.1]{kellner2015}
	Let $\kernel (B'^T)\cap\R^l_+ \neq \{0\}$.
	Then $\pi(P)\subseteq\pi(Q)$ if and only if
	\begin{align}\label{eq:contain}
	\min\left\{ z^T (b-Bx) \ |\ 
	\pi(P)\times\left(\kernel(B'^T)\cap\Delta^{l}\right) \right\} \ge 0 \,,
	\end{align}
	where $\Delta^{l} = \{z\in\R^{l}\ |\ \mathds{1}_{l}^T z = 1,\ z\geq 0\}$ is 
	the $l$-simplex.
\end{prop}

Sufficient for condition $\kernel (B'^T)\cap\R^l_+ \neq \{0\}$ being satisfied 
is boundedness of $\pi(Q)$.
If $\kernel (B'^T)\cap\R^l_+ = \{0\}$, then the simplex $\Delta^{l}$ 
in~\eqref{eq:contain} has to be replaced by the nonnegative orthant $\R^l_+$, 
and thus the feasible set of the bilinear program is unbounded; 
see~\cite[Theorem 3.1]{kellner2015}.
If $m=0$, we can apply Theorem~\ref{thm:fin-convergence}.

\begin{thm}
Let $P\subseteq\R^d$ be a nonempty polytope with $m=0$ and 
$\kernel (B'^T)\cap\Delta^l \neq\emptyset$.
Assume that every optimal solution to~\eqref{eq:contain} is a pair of vertices 
of $P$ and $\kernel(B'^T)\cap\Delta^{l}$.
Then hierarchy~\eqref{eq:putinar} decides containment in finitely many steps.
\end{thm}

As a special case the theorem covers $\HH$-in-$\VV$ containment since every 
$\VV$-polytope can be represented as a projection of an $\HH$-polytope.

\subsection{(Constrained) bimatrix games}\label{sec:game}
A \emph{bimatrix game} $(A,B)$ is given by a pair of \emph{payoff matrices} 
$A, B\in\R^{m\times n}$ and the sets of \emph{mixed strategies} for player 1, 
$\Delta_m = \{x\in\R^m\ |\ \mathds{1}_m^T x=1,\ x\ge 0 \}$, and player 2, 
$\Delta_n = \{y\in\R^n\ |\ \mathds{1}_n^T y=1,\ y\ge 0 \}$.
The payoffs are given by $x^T A y$ for player 1 and by $x^T B y$ for player 2.
Both players want to maximize their payoff.
A pair of mixed strategies $(\bar{x},\bar{y})$ is called a 
\emph{(Nash) equilibrium} if
\[
	\bar{x}^T A\bar{y} \ge x^T A \bar{y}\ \forall x\in\Delta_m
	\quad\text{ and }\quad
	\bar{x}^T B \bar{y} \ge \bar{x}^T B y\ \forall y\in\Delta_n \,.
\]
Mangasarian showed that the set of equilibria corresponds to a subset of the 
vertices of certain polyhedra and can be computed via quadratic programming; 
see~\cite{mangasarian1964,mangasarian1964two}.
As adding a positive constant to every entry of the matrices $A,B$ does not 
change the set of equilibria (see, e.g.,~\cite{vonstengel2002}), we assume that 
$A,B$ are entrywise positive.
Then using a projective transformation boundedness of the polyhedra can be 
achieved.
Define
\[
	P_1 := \{ x\in\R^m\ |\ B^T x \le\mathds{1}_n,\ x\ge 0\} \text{ and }
	P_2 := \{ y\in\R^n\ |\ Ay \le\mathds{1}_m,\ y\ge 0\} \,.
\]

\begin{prop}\cite{mangasarian1964two}\label{prop:manga}
Let $A, B\in\R^{m\times n}$ have only positive entries.
A pair $(\bar{x},\bar{y})$ solves the bilinear 	program
\begin{align}
\begin{split}\label{eq:manga}
	\max\ &\ x^T (A+B) y - x^T \mathds{1}_m - \mathds{1}_n^T y \\
	\ &\ x\in P_1,\ y\in P_2
\end{split}
\end{align}
if and only if $(\frac{\bar{x}}{\mathds{1}_m^T 
\bar{x}},\frac{\bar{y}}{\mathds{1}^T_n \bar{y}})\in\Delta_m\times \Delta_n$
is an equilibrium point for the bimatrix game $(A,B)$.
\end{prop}

A bimatrix game is called \emph{nondegenerate} if in any basic feasible 
solution to the system 
\[
	Ay+s=\mathds{1}_m,\ B^T x+t=\mathds{1}_n,\ x,y,s,t\ge 0
\]
all basic variables have positive values.
(See the survey by von Stengel~\cite{vonstengel2002} for equivalent definitions 
of nondegeneracy.)
In the case of a nondegenerate bimatrix game, the number of equilibria is 
finite as they only appear at vertices of $P_1\times P_2$; 
see~\cite{vonstengel2002}.

\begin{thm}\label{thm:game}
	Let $(A,B)$ be a nondegenerate bimatrix game.
	Then a equilibrium can be computed by hierarchy~\eqref{eq:putinar} in 
	finitely many steps.
\end{thm}

\begin{proof}
By the nondegeneracy assumption, the polytopes in~\eqref{eq:manga} are 
full-dimensional and all optima of the bilinear problem are vertices.
Thus we can apply Theorem~\ref{thm:fin-convergence}.
\end{proof}

Note that the bilinear formulation of bimatrix games can be extended to a 
multilinear formulation of finite $l$-person games; see Mills~\cite{mills1960}.
However, other than bimatrix games, the geometry of $l$-person games is not 
well understood.
In \emph{constrained} bimatrix games, the probability simplices are replaced by 
arbitrary polytopes.
A formulation as a bilinear program is along the same lines as for the 
non-constrained case.
In both cases, an $l$-person game and a constrained game, as long as the number 
of equilibria is known to be finite, Theorem~\ref{thm:game} holds true.


\bibliography{multilinear}
\bibliographystyle{plain}

\end{document}